\documentclass[11pt]{article}
\usepackage{amsthm,amstext}
\usepackage{amsmath}
\usepackage{graphicx}
\usepackage{amsfonts}
\usepackage{amssymb}
\theoremstyle{plain}
\newtheorem{theorem}{Theorem}[section]
\newtheorem{lemma}[theorem]{Lemma}
\newtheorem{proposition}[theorem]{Proposition}

\theoremstyle{definition}
\newtheorem{definition}[theorem]{Definition}

\theoremstyle{remark}
\newtheorem{remark}{\sc Remark}

\date{}
\title{\bf Intuitionistic Fuzzy Points in Semigroups}\vspace{.25 in}
\author{ {\bf Sujit Kumar Sardar$^1$}, {\bf Manasi Mandal$^2$} \\and \\{\bf Samit Kumar Majumder$^3$}\\
Department of Mathematics, Jadavpur\\
University, Kolkata-700032, INDIA\\
{\tt $^1$sksardarjumath@gmail.com}\\
{\tt $^2$manasi$_{-}$ju@yahoo.in}\\
{\tt $^3$samitfuzzy@gmail.com}
 }

\begin{document}
\maketitle

\begin{abstract}

The notion of intuitionistic fuzzy sets was introduced by \textit{Atanassov} as a generalization of the notion of fuzzy sets. {\it Y.B. Jun} and {\it S.Z. Song} introduced the notion of intuitionistic fuzzy points. In this paper we find some relations between the intuitionistic fuzzy ideals of a semigroup $S$ and the set of all intuitionistic fuzzy points of $S.$\\


\end{abstract}

\section{Introduction}
A  semigroup is an algebraic structure consisting of a non-empty set $S$ together with an associative binary operation\cite{H}. The formal study of semigroups began in the early 20th century. Semigroups are important in many areas of mathematics, for example, coding and language theory, automata theory, combinatorics and mathematical analysis. The concept of fuzzy sets was introduced by {\it Lofti Zadeh}\cite{Z} in his classic paper in 1965. {\it Azirel Rosenfeld}\cite{R} used the idea of fuzzy set to introduce the notions of fuzzy subgroups. {\it Nobuaki Kuroki}\cite{K1,K2,K3} is the pioneer of fuzzy ideal theory of semigroups. The idea of fuzzy subsemigroup was also introduced by {\it Kuroki}\cite{K1,K3}. In \cite{K2}, {\it Kuroki} characterized several classes of semigroups in terms of fuzzy left, fuzzy right and fuzzy bi-ideals. Others who worked on fuzzy semigroup theory, such as {\it X.Y. Xie}\cite{X1,X2}, {\it Y.B. Jun}\cite{J1,J2}, are mentioned in the bibliography. {\it X.Y. Xie}\cite{X1} introduced the idea of extensions of fuzzy ideals in semigroups. The notion of intuitionistic fuzzy sets was introduced by {\it Atanassov}\cite{A1,A2,A3} as a generalization of the notion of fuzzy sets. {\it Pu} and {\it Liu}\cite{P} introduced the notion of fuzzy points. In \cite{W}, {\it X.P. Wang, Z.W. Mo, W.J. Liu}, in \cite{Y} {\it Y.H. Yon} and in \cite{K} {\it K.H. Kim} characterized fuzzy ideals as fuzzy points of semigroups. {\it Y.B. Jun} and {\it S.Z. Song} introduced the notion of intuitionistic fuzzy points\cite{J2}. In this paper, we consider the semigroup $\underline{S}$ of the intuitionistic fuzzy points of a semigroup $S,$ and discuss some relations between the fuzzy subsemigroups$($fuzzy bi-ideals, fuzzy interior ideals, fuzzy ideals, fuzzy prime ideals, fuzzy semiprime ideals$)$ of $S$ and the subsets of $\underline{S}.$ Among other results we obtain some characterization theorems of regular and intra-regular semigroups in terms of intuitionistic fuzzy points.
\section{Preliminaries}

In this section we discuss some elementary definitions that we use
in the sequel.\\

\begin{definition}
\cite{Mo} If $(X,\ast)$ is a mathematical system such that $\forall a,b,c\in X,$ $(a\ast b)\ast c=a\ast(b\ast c),$ then $\ast$ is called associative and $(X,\ast)$ is called a {\it semigroup}.
\end{definition}

\begin{definition}
\cite{Mo} A {\it subsemigroup} of a semigroup $S$ is a non-empty subset $I$ of $S$ such that $I^{2} \subseteq I.$
\end{definition}

\begin{definition}
\cite{Mo} A subsemigroup $I$ of a semigroup $S$ is a called an {\it interior ideal} of $S$ if $SIS\subseteq I.$
\end{definition}

\begin{definition}
\cite{Mo} A subsemigroup $I$ of a semigroup $S$ is a called a {\it bi-ideal} of $S$ if $ISI\subseteq I.$
\end{definition}

\begin{definition}
\cite{Mo} A {\it left} ({\it right})
{\it ideal} of a semigroup $S$ is a non-empty subset $I$ of $S$ such that $SI \subseteq I$ ($IS \subseteq I$). If $I$ is both a left and a right ideal of a semigroup $S$, then we say that $I$ is an {\it ideal} of $S$.
\end{definition}

\begin{definition}
\cite{Mo} Let $S$ be a semigroup. Then an ideal $I$ of $S$ is said to be $(i)$ {\it prime} if for ideals $A,B$ of $S,AB\subseteq I$ implies that $A\subseteq I$ or $B\subseteq I,(ii)$ {\it semiprime} if for an ideal $A$ of $S,A^{2}\subseteq I$ implies that $A\subseteq I$.
\end{definition}

\begin{definition}
\cite{A1,A2} The {\it intuitionistic fuzzy sets} defined on a non-empty set $X$ as objects having the form\\
$$A=\{<x,\mu_{A}(x),\nu_{A}(x)>:x\in X\},$$
where the functions $\mu_{A}: X\rightarrow [0,1]$ and $\nu_{A}: X\rightarrow [0,1]$ denote the degree of membership and the degree of non-membership of each element $x\in X$ to the set $A$ respectively, and $0\leq \mu_{A}(x)+\nu_{A}(x)\leq 1$ for all $x\in X.$

For the sake of simplicity, we shall use the symbol $A=(\mu_{A},\nu_{A})$ for the intuitionistic fuzzy subset $A=\{<x,\mu_{A}(x),\nu_{A}(x)>:x\in X\}.$
\end{definition}

\begin{definition}
\cite{J2} Let $\alpha,\beta\in [0,1]$ with $\alpha+\beta\leq 1.$ An {\it intuitionistic fuzzy point}, written as $x_{(\alpha,\beta)},$ is defined to be an intuitionistic fuzzy subset of $S,$ given by

\begin{align*}
x_{(\alpha,\beta)} (y)=\left\{
\begin{array}{ll}
(\alpha,\beta) & \text{if} \ x=y \\
(0,1) & \text{otherwise}
\end{array}
\right.
\end{align*}
\end{definition}

\begin{definition}
\cite{S} A non-empty intuitionistic fuzzy subset $A=(\mu_{A},\nu_{A})$ of a semigroup $S$ is called an {\it intuitionistic fuzzy subsemigroup} of $S$ if $(i)$ $\mu_{A}(xy)\geq\min\{\mu_{A}(x),\mu_{A}\\(y)\}$ $\forall x,y\in S,$ $(ii)$ $\nu_{A}(xy)\leq\max\{\nu_{A}(x),\nu_{A}(y)\}$ $\forall x,y\in S.$
\end{definition}

\begin{definition}
\cite{S} An intuitionistic fuzzy subsemigroup $A=(\mu_{A},\nu_{A})$ of a semigroup $S$ is called an {\it intuitionistic fuzzy interior ideal} of $S$ if $(i)$ $\mu_{A}(xay)\geq\mu_{A}(a)$ $\forall x,a,y\in S,$ $(ii)$ $\nu_{A}(xay)\leq\nu_{A}(a)$ $\forall x,a,y\in S.$
\end{definition}

\begin{definition}
\cite{S} An intuitionistic fuzzy subsemigroup $A=(\mu_{A},\nu_{A})$ of a semigroup $S$ is called an {\it intuitionistic fuzzy bi-ideal} of $S$ if $(i)$ $\mu_{A}(x\omega y)\geq\min\{\mu_{A}(x),\mu_{A}(y)\}$ $\forall x,\omega,y\in S,$ $(ii)$ $\nu_{A}(x\omega y)\leq\max\{\nu_{A}(x),\nu_{A}(y)\}$ $\forall x,\omega,y\in S.$
\end{definition}

\begin{definition}
\cite{S} A non-empty intuitionistic fuzzy subset $A=(\mu_{A},\nu_{A})$ of a semigroup $S$ is called an {\it intuitionistic fuzzy left$($right$)$ ideal} of $S$ if $(i)$ $\mu_{A}(xy)\geq\mu_{A}(y)($resp. $\mu_{A}(xy)\geq\mu_{A}(x))$\ $\forall x,y\in S,$ $(ii)$ $\nu_{A}(xy)\leq\nu_{A}(y)($resp. $\nu_{A}(xy)\leq\nu_{A}(x))$\ $\forall x,y\in S.$
\end{definition}

\begin{definition}
\cite{S} A non-empty intuitionistic fuzzy subset $A=(\mu_{A},\nu_{A})$ of a semigroup $S$ is called an {\it intuitionistic fuzzy two-sided ideal} or an {\it intuitionistic fuzzy ideal} of $S$ if it is both an intuitionistic fuzzy left and an intuitionistic fuzzy right ideal of $S.$
\end{definition}
Alternative definition of Definition $2.13,$ is as follows.

\begin{definition}
\cite{S} A non-empty intuitionistic fuzzy subset $A=(\mu_{A},\nu_{A})$ of a semigroup $S$ is called an {\it intuitionistic fuzzy ideal} of $S$ if $(i)$ $\mu_{A}(xy)\geq\max\{\mu_{A}(x),\mu_{A}(y)\}\forall x,\\y\in S,$ $(ii)$ $\nu_{A}(xy)\leq\min\{\nu_{A}(x),\nu_{A}(y)\}\forall x,y\in S.$
\end{definition}

\begin{definition}
\cite{S} An intuitionistic fuzzy ideal $A=(\mu_{A},\nu_{A})$ of a semigroup $S$ is called an {\it intuitionistic fuzzy semiprime ideal} of $S$ if $(i)$ $\mu_{A}(x)\geq\mu_{A}(x^{2})\forall x\in S,$ $(ii)$ $\nu_{A}(x)\leq\nu_{A}(x^{2})$ $\forall x\in S.$
\end{definition}

\begin{definition}
\cite{S} An intuitionistic fuzzy ideal $A=(\mu_{A},\nu_{A})$ of a semigroup $S$ is called an {\it intuitionistic fuzzy prime ideal} of $S$ if $(i)$ $\mu_{A}(xy)=\max\{\mu_{A}(x),\mu_{A}(y)\}\forall x,y\in S,$ $(ii)$ $\nu_{A}(xy)=\min\{\nu_{A}(x),\nu_{A}(y)\}$ $\forall x,y\in S.$
\end{definition}

\section{Main Results}

Let $\Im(S)$ be the set of all intuitionistic fuzzy subsets of a semigroup $S.$ For each $A=(\mu_{A},\nu_{A}),$ $B=(\mu_{B},\nu_{B})\in \Im(S),$ the product of $A$ and $B$ is an intuitionistic fuzzy subset $A\circ B$ defined as follows:

$$A\circ B=\{<x,(\mu_{A}\circ\mu_{B})(x),(\nu_{A}\circ\nu_{B})(x)>:x\in S\}$$

$$
\text{ where } (\mu_{A}\circ \mu_{B})(x)=\left\{
\begin{array}
[c]{c}%
\underset{x=uv}{\sup}[\min\{\mu_{A}(u),\mu_{B}(v)\}:u,v\in S]\\
0,\text{ if for any }u,v\in S\text{ },x\neq uv
\end{array}
\right.\\
$$
$$
 \text{ and  }
(\nu_{A}\circ\nu_{B})(x)=\left\{
\begin{array}
[c]{c}%
\underset{x=uv}{\inf}[\max\{\nu_{A}(u),\nu_{B}(v)\}:u,v\in S]\\
1,\text{ if for any }u,v\in S\text{ },x\neq uv
\end{array}
\right.
$$

It is clear that $(A\circ B)\circ C=A\circ(B\circ C),$ and that if $A\subseteq B,$ then $A\circ C\subseteq B\circ C$ and $C\circ A\subseteq C\circ B$ for any $A,B,C\in \Im(S).$ Thus $\Im(S)$ is a semigroup with the product $\circ.$

\indent Let $\underline{S}$ be the set of all intuitionistic fuzzy points in a semigroup $S.$ Then $x_{(\alpha,\beta)}\circ y_{(\gamma,\delta)}=(xy)_{(\alpha\wedge\gamma,\beta\vee\delta)}\in \underline{S},$ where $\alpha\wedge\gamma=\min(\alpha,\gamma)$ and $\beta\vee\delta=\max(\beta,\delta)$ and $x_{(\alpha,\beta)}\circ (y_{(\gamma,\delta)}\circ z_{(\eta,\theta)})=(xyz)_{(\alpha\wedge\gamma\wedge\eta,\beta\vee\delta\vee\theta)}=(x_{(\alpha,\beta)}\circ y_{(\gamma,\delta)})\circ z_{(\eta,\theta)}$ for any $x_{(\alpha,\beta)},y_{(\gamma,\delta)},z_{(\eta,\theta)}\in\underline{S}.$ Thus $\underline{S}$ is a subsemigroup of $\Im(S).$

For any $A=(\mu_{A},\nu_{A})\in\Im(S),$ $\underline{A}$ denote the set of all intuitionistic fuzzy points contained in $A=(\mu_{A},\nu_{A})$, that is, $\underline{A}=\{x_{(\alpha,\beta)}\in\underline{S}:\mu_{A}(x)\geq\alpha$ and $\nu_{A}(x)\leq\beta\}.$ If $x_{(\alpha,\beta)}\in\underline{S},$ then $\alpha>0$ and $\beta<1.$

\begin{proposition}
Let $A=(\mu_{A},\nu_{A})$ and $B=(\mu_{B},\nu_{B})$ be two intuitionistic fuzzy subsets of a semigroup $S.$ Then

$(i)$ $\underline{A\cup B}=\underline{A}\cup\underline{B}.$

$(ii)$ $\underline{A\cap B}=\underline{A}\cap\underline{B}.$

$(iii)$ $\underline{A\circ B}\supseteq\underline{A}\circ\underline{B}.$

\end{proposition}

\begin{proof} \begin{eqnarray*}
(i)\text{ Let }x_{(\alpha ,\beta )} &\in &\underline{A\cup B} \\
&\Leftrightarrow &\{x\in \underline{S}:(\mu _{A}\cup \mu _{B})(x)\geq \alpha
\text{ and }(\nu _{A}\cap \nu _{B})(x)\leq \beta \}\text{ } \\
&\Leftrightarrow &\{x\in \underline{S}:\max \{\mu _{A}(x),\mu _{B}(x)\}\geq
\alpha \text{  and }\min \{\nu _{A}(x),\nu _{B}(x)\}\leq \beta \} \\
&\Leftrightarrow &[x\in \underline{S}:\{\mu _{A}(x)\geq \alpha \text{ or }%
\mu _{B}(x)\geq \alpha \}\text{  and }\{\nu _{A}(x)\leq \beta \text{ or }\nu
_{B}(x)\leq \beta \} \\
&\Leftrightarrow &\{x\in \underline{S}:\mu _{A}(x)\geq \alpha \text{ and }%
\nu _{A}(x)\leq \beta \}\text{ or }\\
& &\{x\in \underline{S}:\mu _{B}(x)\geq
\alpha \text{ and }\nu _{B}(x)\leq \beta \} \\
&\Leftrightarrow &x_{(\alpha ,\beta )}\in \underline{A}\text{ or }x_{(\alpha ,\beta )}\in \underline{B}\\
&\Leftrightarrow &x_{(\alpha ,\beta )}\in \underline{A}\cup \underline{B}.
\end{eqnarray*}
Hence $\underline{A\cup B}=\underline{A}\cup \underline{B}.$

\begin{eqnarray*}
(ii)\text{ Let }x_{(\alpha ,\beta )} &\in &\underline{A\cap B} \\
&\Leftrightarrow &\{x\in \underline{S}:(\mu _{A}\cap \mu _{B})(x)\geq \alpha
\text{ and }(\nu _{A}\cup \nu _{B})(x)\leq \beta \}\text{ } \\
&\Leftrightarrow &\{x\in \underline{S}:\min \{\mu _{A}(x),\mu _{B}(x)\}\geq
\alpha \text{ and }\max \{\nu _{A}(x),\nu _{B}(x)\}\leq \beta \} \\
&\Leftrightarrow &[x\in \underline{S}:\{\mu _{A}(x)\geq \alpha \text{ and }%
\mu _{B}(x)\geq \alpha \}\text{ and }\\
& &\{\nu _{A}(x)\leq \beta \text{ and }\nu
_{B}(x)\leq \beta \}] \\
&\Leftrightarrow &\{x\in \underline{S}:\mu _{A}(x)\geq \alpha \text{ and }%
\nu _{A}(x)\leq \beta \}\text{ and }\\
& &\{x\in \underline{S}:\mu _{B}(x)\geq
\alpha \text{ and }\nu _{B}(x)\leq \beta \} \\
&\Leftrightarrow &x_{(\alpha ,\beta )}\in \underline{A}\text{ and }%
x_{(\alpha ,\beta )}\in \underline{B} \\
&\Leftrightarrow &x_{(\alpha ,\beta )}\in \underline{A}\cap \underline{B}.
\end{eqnarray*}
Hence $\underline{A\cap B}=\underline{A}\cap \underline{B}.$

$(iii)$
\begin{eqnarray*}
\underline{A}\circ \underline{B}&=&\{x_{(\alpha,\beta)}\circ y_{(\gamma,\delta)}:x_{(\alpha,\beta)}\in\underline{A},y_{(\gamma,\delta)}\in\underline{B}\}\\
&=&\{(xy)_{(\alpha\wedge\gamma,\beta\vee\delta)}:\mu_{A}(x)\geq\alpha\text{ and }\nu_{A}(x)\leq\beta,\mu_{B}(y)\geq\gamma\text{ and }\nu_{B}(y)\leq\delta\}\\
&=&\{(xy)_{(\alpha\wedge\gamma,\beta\vee\delta)}:\mu_{A}(x)\geq\alpha\text{ and }\mu_{B}(y)\geq\gamma,\nu_{B}(x)\leq\beta\text{ and }\nu_{B}(y)\leq\delta\}\\
&=&\{(xy)_{(\alpha\wedge\gamma,\beta\vee\delta)}:\min\{\mu_{A}(x),\mu_{B}(y)\}\geq\alpha\wedge\gamma,\max\{\nu_{B}(x),\nu_{B}(y)\}\leq\beta\vee\delta\}\\
&\leq &\{(xy)_{(\alpha\wedge\gamma,\beta\vee\delta)}:\underset{%
\begin{array}{c}
x,y\in S \\
\mu _{A}(x)\geq \alpha \\
\mu _{B}(y)\geq \gamma%
\end{array}%
}{\sup }[\min \{\mu _{A}(x),\mu _{B}(y)\}]\geq \alpha \wedge \gamma,\\
&&\underset{%
\begin{array}{c}
x,y\in S \\
\nu _{A}(x)\leq \beta \\
\nu _{B}(y)\leq \delta%
\end{array}
}{\inf }[\max \{\nu _{A}(x),\nu _{B}(y)\}]\leq \beta \vee \delta\}\\
&=& \{(u)_{(\alpha\wedge\gamma,\beta\vee\delta)}:(\mu_{A}\circ\mu_{B})(u)\geq\alpha\wedge\gamma,(\nu_{B}\circ\nu_{B})(u)\leq\beta\vee\delta\}\\
&=&\underline{A\circ B}.
\end{eqnarray*}

Hence $\underline{A}\circ \underline{B}\subseteq\underline{A\circ B}.$
\end{proof}

\begin{theorem}
Let $A=(\mu_{A},\nu_{A})$ be a non-empty intuitionistic fuzzy subset of a semigroup $S.$ Then following conditions are equivalent:

$(i)$ $A=(\mu_{A},\nu_{A})$ is an intuitionistic fuzzy subsemigroup of $S,$

$(ii)$ $\underline{A}$ is a subsemigroup of $\underline{S}.$
\end{theorem}

\begin{proof}
$(i)\Rightarrow (ii):$ Let $A=(\mu_{A},\nu_{A})$ be an intuitionistic fuzzy subsemigroup of $S.$ Let $x_{(\alpha,\beta)},y_{(\gamma,\delta)}\in\underline{A}.$ Then $\mu_{A}(x)\geq\alpha>0,$ $\mu_{A}(y)\geq\gamma>0$ and $\nu_{A}(x)\leq\beta<1,$ $\nu_{A}(y)\leq\delta<1.$ Since $\mu_{A}(xy)\geq\min\{\mu_{A}(x),\mu_{A}(y)\}\geq\alpha\wedge\gamma$ and $\nu_{A}(xy)\leq\max\{\nu_{A}(x),\nu_{A}(y)\}\leq\beta\vee\delta.$ Consequently, $x_{(\alpha,\beta)}\circ y_{(\gamma,\delta)}=(xy)_{(\alpha\wedge\gamma,\beta\vee\delta)}\in\underline{A}.$ This implies that $\underline{A}^{2}\subseteq\underline{A}.$ Hence $\underline{A}$ is a subsemigroup of $\underline{S}.$

$(ii)\Rightarrow (i):$ Let us suppose that $\underline{A}$ is a subsemigroup of $\underline{S}.$ Let $x,y\in S.$ If $\mu_{A}(x)=\mu_{A}(y)=0$ and $\nu_{A}(x)=\nu_{A}(y)=1,$ then $\min\{\mu_{A}(x),\mu_{A}(y)\}=0\leq\mu_{A}(xy)$ and $\max\{\nu_{A}(x),\nu_{A}(y)\}=1\geq\nu_{A}(xy).$ If $\mu_{A}(x)=\mu_{A}(y)\neq 0$ and $\nu_{A}(x)=\nu_{A}(y)<1,$ then $x_{(\mu_{A}(x),\nu_{A}(x))},y_{(\mu_{A}(y),\nu_{A}(y))}\in\underline{A}.$ Since $\underline{A}$ is a subsemigroup of $\underline{S},$ so we have $(xy)_{(\mu_{A}(x)\wedge\mu_{A}(y),\nu_{A}(x)\vee\nu_{A}(y))}=x_{(\mu_{A}(x),\nu_{A}(x))}\circ y_{(\mu_{A}(y),\nu_{A}(y))}\in \underline{A}$ which implies that $\mu_{A}(xy)\\\geq\mu_{A}(x)\wedge\mu_{A}(y)=\min\{\mu_{A}(x),\mu_{A}(y)\}$ and $\nu_{A}(xy)\leq\nu_{A}(x)\vee\nu_{A}(y)=\max\{\nu_{A}(x),\nu_{A}(y)\}.$ Hence $A=(\mu_{A},\nu_{A})$ be an intuitionistic fuzzy subsemigroup of $S.$
\end{proof}

\begin{theorem}
Let $A=(\mu_{A},\nu_{A})$ be a non-empty intuitionistic fuzzy subset of a semigroup $S.$ Then following conditions are equivalent:

$(i)$ $A=(\mu_{A},\nu_{A})$ is an intuitionistic fuzzy bi-ideal of $S,$

$(ii)$ $\underline{A}$ is a bi-ideal of $\underline{S}.$
\end{theorem}

\begin{proof}
$(i)\Rightarrow (ii):$ Let $A=(\mu_{A},\nu_{A})$ be an intuitionistic fuzzy bi-ideal of $S.$ Then $A=(\mu_{A},\nu_{A})$ is an intuitionistic fuzzy subsemigroup of $S.$ Hence by Theorem $3.2,$ $\underline{A}$ is a subsemigroup of $\underline{S}.$ Let $x_{(\alpha,\beta)},z_{(\eta,\theta)}\in\underline{A}$ and $y_{(\gamma,\delta)}\in\underline{S}.$ Then $\mu_{A}(x)\geq\alpha>0,$ $\mu_{A}(z)\geq\eta>0$ and $\nu_{A}(x)\leq\beta<1,$ $\nu_{A}(z)\leq\theta<1.$ Since $\mu_{A}(xyz)\geq\min\{\mu_{A}(x),\mu_{A}(z)\}\geq\alpha\wedge\gamma\wedge\eta($since $\gamma>0)$ and $\nu_{A}(xyz)\leq\max\{\nu_{A}(x),\nu_{A}(z)\}\leq\beta\vee\delta\vee\theta($since $\delta<1).$ Consequently, $x_{(\alpha,\beta)}\circ y_{(\gamma,\delta)}\circ z_{(\eta,\theta)}=(xyz)_{(\alpha\wedge\gamma\wedge\eta,\beta\vee\delta\vee\theta)}\in\underline{A}.$ This implies that $\underline{A}$ $\underline{S}$ $\underline{A}\subseteq\underline{A}.$ Hence $\underline{A}$ is a bi-ideal of $\underline{S}.$

$(ii)\Rightarrow (i):$ Let us suppose that $\underline{A}$ is a bi-ideal of $\underline{S}.$ Then $\underline{A}$ is a subsemigroup of $\underline{S}.$ Hence by Theorem $3.2,$ $A=(\mu_{A},\nu_{A})$ be an intuitionistic fuzzy subsemigroup of $S.$ Let $x,y,z\in S.$ If $\mu_{A}(x)=\mu_{A}(z)=0$ and $\nu_{A}(x)=\nu_{A}(z)=1,$ then $\min\{\mu_{A}(x),\mu_{A}(z)\}=0\leq\mu_{A}(xyz)$ and $\max\{\nu_{A}(x),\nu_{A}(z)\}=1\geq\nu_{A}(xyz).$ If $\mu_{A}(x)=\mu_{A}(z)\neq 0$ and $\nu_{A}(x)=\nu_{A}(z)<1,$ then $x_{(\mu_{A}(x),\nu_{A}(x))},z_{(\mu_{A}(z),\nu_{A}(z))}\in\underline{A}.$ Since $\underline{A}$ is a bi-ideal of $\underline{S},$ so we have $(xyz)_{(\mu_{A}(x)\wedge\mu_{A}(z),\nu_{A}(x)\vee\nu_{A}(z))}=(xyz)_{(\mu_{A}(x)\wedge\mu_{A}(x)\wedge\mu_{A}(z),\nu_{A}(x)\vee\nu_{A}(x)\vee\nu_{A}(z))}=x_{(\mu_{A}(x),\nu_{A}(x))}\circ y_{(\mu_{A}(x),\nu_{A}(x))}\circ z_{(\mu_{A}(z),\nu_{A}(z))}\in \underline{A}$ which implies that $\mu_{A}(xyz)\geq\mu_{A}(x)\wedge\mu_{A}(z)=\min\{\mu_{A}(x),\mu_{A}(z)\}$ and $\nu_{A}(xyz)\leq\nu_{A}(x)\vee\nu_{A}(z)=\max\{\nu_{A}(x),\nu_{A}(z)\}.$ Hence $A=(\mu_{A},\nu_{A})$ be an intuitionistic fuzzy bi-ideal of $S.$
\end{proof}

\begin{theorem}
Let $A=(\mu_{A},\nu_{A})$ be a non-empty intuitionistic fuzzy subset of a semigroup $S.$ Then following conditions are equivalent:

$(i)$ $A=(\mu_{A},\nu_{A})$ is an intuitionistic fuzzy interior ideal of $S,$

$(ii)$ $\underline{A}$ is an interior ideal of $\underline{S}.$
\end{theorem}

\begin{proof}
$(i)\Rightarrow (ii):$ Let $A=(\mu_{A},\nu_{A})$ be an intuitionistic fuzzy interior ideal of $S.$ Then $A=(\mu_{A},\nu_{A})$ is an intuitionistic fuzzy subsemigroup of $S.$ Hence by Theorem $3.2,$ $\underline{A}$ is a subsemigroup of $\underline{S}.$ Let $x_{(\alpha,\beta)},z_{(\eta,\theta)}\in\underline{S}$ and $y_{(\gamma,\delta)}\in\underline{A}.$ Then $\mu_{A}(y)\geq\gamma>0$ and $\nu_{A}(y)\leq\delta<1.$ Since $\mu_{A}(xyz)\geq\mu_{A}(y)\geq\alpha\wedge\gamma\wedge\eta($since $\alpha,\eta>0)$ and $\nu_{A}(xyz)\leq\nu_{A}(y)\leq\beta\vee\delta\vee\theta($since $\beta,\theta<1).$ Consequently, $x_{(\alpha,\beta)}\circ y_{(\gamma,\delta)}\circ z_{(\eta,\theta)}=(xyz)_{(\alpha\wedge\gamma\wedge\eta,\beta\vee\delta\vee\theta)}\in\underline{A}.$ This implies that $\underline{S}$ $\underline{A}$ $\underline{S}\subseteq\underline{A}.$ Hence $\underline{A}$ is an interior ideal of $\underline{S}.$

$(ii)\Rightarrow (i):$ Let us suppose that $\underline{A}$ is an interior ideal of $\underline{S}.$ Then $\underline{A}$ is a subsemigroup of $\underline{S}.$ Hence by Theorem $3.2,$ $A=(\mu_{A},\nu_{A})$ be an intuitionistic fuzzy subsemigroup of $S.$ Let $x,y,z\in S.$ If $\mu_{A}(y)=0$ and $\nu_{A}(y)=1,$ then $\mu_{A}(y)=0\leq\mu_{A}(xyz)$ and $\nu_{A}(y)=1\geq\nu_{A}(xyz).$ If $\mu_{A}(y)\neq 0$ and $\nu_{A}(y)<1,$ then $y_{(\mu_{A}(y),\nu_{A}(y))}\in\underline{A}.$ Since $\underline{A}$ is an interior ideal of $\underline{S},$ so we have $(xyz)_{(\mu_{A}(y),\nu_{A}(y))}=(xyz)_{(\mu_{A}(y)\wedge\mu_{A}(y)\wedge\mu_{A}(y),\nu_{A}(y)\vee\nu_{A}(y)\vee\nu_{A}(y))}=x_{(\mu_{A}(y),\nu_{A}(y))}\circ y_{(\mu_{A}(y),\nu_{A}(y))}\circ z_{(\mu_{A}(y),\nu_{A}(y))}\in \underline{A}$ which implies that $\mu_{A}(xyz)\geq\mu_{A}(y)$ and $\nu_{A}(xyz)\leq\nu_{A}(y).$ Hence $A=(\mu_{A},\nu_{A})$ be an intuitionistic fuzzy interior ideal of $S.$
\end{proof}

\begin{theorem}
Let $A=(\mu_{A},\nu_{A})$ be a non-empty intuitionistic fuzzy subset of a semigroup $S.$ Then following conditions are equivalent:

$(i)$ $A=(\mu_{A},\nu_{A})$ is an intuitionistic fuzzy $($right, left$)$ ideal of $S,$

$(ii)$ $\underline{A}$ is a right$($left$)$ ideal of $\underline{S}.$
\end{theorem}

\begin{proof}
$(i)\Rightarrow (ii):$ Let $A=(\mu_{A},\nu_{A})$ be an intuitionistic fuzzy right ideal of $S.$ Let $x_{(\alpha,\beta)}\in\underline{A}$ and $y_{(\gamma,\delta)}\in\underline{S}.$ Then $\mu_{A}(x)\geq\alpha>0$ and $\nu_{A}(x)\leq\beta<1.$ Since $\mu_{A}(xy)\geq\mu_{A}(x)\geq\alpha\wedge\gamma($since $\gamma>0)$ and $\nu_{A}(xy)\leq\nu_{A}(x)\leq\beta\vee\delta($since $\delta<1).$ Consequently, $x_{(\alpha,\beta)}\circ y_{(\gamma,\delta)}=(xy)_{(\alpha\wedge\gamma,\beta\vee\delta)}\in\underline{A}.$ This implies that $\underline{A}$ $\underline{S}$ $\subseteq\underline{A}.$ Hence $\underline{A}$ is a right ideal of $\underline{S}.$

$(ii)\Rightarrow (i):$ Let us suppose that $\underline{A}$ is a right ideal of $\underline{S}.$ Let $x,y\in S.$ If $\mu_{A}(x)=0$ and $\nu_{A}(x)=1,$ then $\mu_{A}(x)=0\leq\mu_{A}(xy)$ and $\nu_{A}(x)=1\geq\nu_{A}(xy).$ If $\mu_{A}(x)\neq 0$ and $\nu_{A}(x)<1,$ then $x_{(\mu_{A}(x),\nu_{A}(x))}\in\underline{A}.$ Since $\underline{A}$ is a right ideal of $\underline{S},$ so we have $(xy)_{(\mu_{A}(x)\wedge\mu_{A}(x),\nu_{A}(x)\vee\nu_{A}(x))}=x_{(\mu_{A}(x),\nu_{A}(x))}\circ y_{(\mu_{A}(x),\nu_{A}(x))}\in \underline{A}$ which implies that $\mu_{A}(xy)\geq\mu_{A}(x)$ and $\nu_{A}(xy)\leq\nu_{A}(x).$ Hence $A=(\mu_{A},\nu_{A})$ be an intuitionistic fuzzy right ideal of $S.$ Similarly we can prove the theorem for left ideal and ideal.
\end{proof}

\begin{remark}
It is clear that any ideal of a semigroup $S$ is an interior ideal of $S.$ It is also clear that any intuitionistic fuzzy ideal of a semigroup $S$ is an intuitionistic fuzzy interior ideal of $S.$
\end{remark}

\begin{definition}
$\cite{K2}$ A semigroup $S$ is called regular if, for each element $a$ of $S,$ there exists an element $x\in S$ such that $a=axa.$ A semigroup $S$ is called intra-regular if, for each element $x\in S,$ there exist elements $a,b\in S$ such that $x=ax^{2}b.$
\end{definition}

\begin{theorem}
Let $A=(\mu_{A},\nu_{A})$ be a non-empty intuitionistic fuzzy subset of a regular semigroup $S.$ Then following conditions are equivalent:

$(i)$ $A=(\mu_{A},\nu_{A})$ is an intuitionistic fuzzy $($right, left$)$ ideal of $S,$

$(ii)$ $\underline{A}$ is an interior ideal of $\underline{S}.$
\end{theorem}

\begin{proof}
$(i)\Rightarrow (ii):$ Follows easily from Remark $1$ and Theorem $3.4.$

$(ii)\Rightarrow (i):$ Let us suppose that $\underline{A}$ is an interior ideal of $\underline{S}.$ Let $x\in S.$ Then there exists an element $a\in S$ such that $x=xax($since $S$ is regular$).$ If $\mu_{A}(x)=0$ and $\nu_{A}(x)=1,$ then $\mu_{A}(x)=0\leq\mu_{A}(xy)$ and $\nu_{A}(x)=1\geq\nu_{A}(xy).$ If $\mu_{A}(x)\neq 0$ and $\nu_{A}(x)<1,$ then $x_{(\mu_{A}(x),\nu_{A}(x))}\in\underline{A}$ and $y_{(\mu_{A}(x),\nu_{A}(x))}\in\underline{S}.$ Since $\underline{A}$ is an interior ideal of $\underline{S},$ so we have $(xy)_{(\mu_{A}(x),\nu_{A}(x))}=(xaxy)_{(\mu_{A}(x),\nu_{A}(x))}=((xa)xy)_{(\mu_{A}(x)\wedge\mu_{A}(x)\wedge\mu_{A}(x),\nu_{A}(x)\vee\nu_{A}(x)\vee\nu_{A}(x))}=(xa)_{(\mu_{A}(x),\nu_{A}(x))}\circ x_{(\mu_{A}(x),\nu_{A}(x))}\circ y_{(\mu_{A}(x),\nu_{A}(x))}\in \underline{A}$ which implies that $\mu_{A}(xy)\geq\mu_{A}(x)$ and $\nu_{A}(xy)\leq\nu_{A}(x).$ Hence $A=(\mu_{A},\nu_{A})$ be an intuitionistic fuzzy right ideal of $S.$ The proof is similar in case of intuitionistic fuzzy left ideal and intuitionistic fuzzy ideal.
\end{proof}

\begin{theorem} Let $S$ be a semigroup. If for fixed $\alpha,\beta\in [0,1]$ with $\alpha+\beta\leq 1,$ $f_{(\alpha,\beta)}:S\rightarrow \underline{S}$ is a function defined by $f_{(\alpha,\beta)}(x)=x_{(\alpha,\beta)},$ then $f_{(\alpha,\beta)}$ is an injective homomorphism.
\end{theorem}

\begin{proof}
\underline{To show $f_{(\alpha,\beta)}$ is a homomorphism:}
Let $x,y\in S.$ Then%
\begin{align*}
f_{(\alpha,\beta)}(xy) & =(xy)_{(\alpha,\beta)}=(xy)_{(\alpha\wedge\alpha,\beta\vee\beta)}=x_{(\alpha,\beta)}\circ y_{(\alpha,\beta)}=f_{(\alpha,\beta)}(x)f_{(\alpha,\beta)}(y).
\end{align*}
Hence $f_{(\alpha,\beta)}$ is an homomorphism.\\

\underline{To show $f_{(\alpha,\beta)}$ is injective:}
Let $x_{1},x_{2}\in S.$ Then $f_{(\alpha,\beta)}(x_{1}) =f_{(\alpha,\beta)}(x_{2})\Rightarrow (x_{1})_{(\alpha,\beta)} =(x_{2})_{(\alpha,\beta)}\Rightarrow x_{1}=x_{2}.$

Hence $f_{(\alpha,\beta)}$ is an injective homomorphism.
\end{proof}

\begin{theorem}
Let $A=(\mu_{A},\nu_{A})$ be a non-empty intuitionistic fuzzy subset of an intra-regular semigroup $S.$ Then following conditions are equivalent:

$(i)$ $A=(\mu_{A},\nu_{A})$ is an intuitionistic fuzzy $($right, left$)$ ideal of $S,$

$(ii)$ $\underline{A}$ is an interior ideal of $\underline{S}.$
\end{theorem}

\begin{proof}
$(i)\Rightarrow (ii):$ Follows easily from Remark $1$ and Theorem $3.4.$

$(ii)\Rightarrow (i):$ Let us suppose that $\underline{A}$ is an interior ideal of $\underline{S}.$ Let $x,y\in S.$ Then there exist elements $a,b\in S$ such that $x=ax^{2}b($since $S$ is intra-regular$).$ If $\mu_{A}(x)=0$ and $\nu_{A}(x)=1,$ then $\mu_{A}(x)=0\leq\mu_{A}(xy)$ and $\nu_{A}(x)=1\geq\nu_{A}(xy).$ If $\mu_{A}(x)\neq 0$ and $\nu_{A}(x)<1,$ then $x_{(\mu_{A}(x),\nu_{A}(x))}\in\underline{A}$ and $y_{(\mu_{A}(x),\nu_{A}(x))}\in\underline{S}.$ Since $\underline{A}$ is an interior ideal of $\underline{S},$ so we have %
\begin{align*}
(xy)_{(\mu_{A}(x),\nu_{A}(x))}&=(ax^{2}by)_{(\mu_{A}(x),\nu_{A}(x))}\\
&=((ax)x(by))_{(\mu_{A}(x)\wedge\mu_{A}(x)\wedge\mu_{A}(x),\nu_{A}(x)\vee\nu_{A}(x)\vee\nu_{A}(x))}\\
& =(ax)_{(\mu_{A}(x),\nu_{A}(x))}\circ x_{(\mu_{A}(x),\nu_{A}(x))}\circ (by)_{(\mu_{A}(x),\nu_{A}(x))}\in \underline{A}.
\end{align*}
This implies that $\mu_{A}(xy)\geq\mu_{A}(x)$ and $\nu_{A}(xy)\leq\nu_{A}(x).$ Hence $A=(\mu_{A},\nu_{A})$ be an intuitionistic fuzzy right ideal of $S.$ The proof is similar in case of intuitionistic fuzzy left ideal and intuitionistic fuzzy ideal.
\end{proof}

\begin{theorem}
A semigroup $S$ is intra-regular if and only if the semigroup $\underline{S}$ is intra-regular.
\end{theorem}

\begin{proof}
Let $a_{(\alpha,\beta)}\in\underline{S}$ and $a\in S.$ Then there exist elements $x,y\in S$ such that $a=xa^{2}y($since $S$ is intra-regular$).$ So $x_{(\alpha,\beta)},y_{(\alpha,\beta)}\in\underline{S}.$ Then%
\begin{align*}
x_{(\alpha,\beta)}\circ a_{(\alpha,\beta)}\circ a_{(\alpha,\beta)}\circ y_{(\alpha,\beta)}& =x_{(\alpha,\beta)}\circ (a^{2})_{(\alpha\wedge\alpha,\beta\vee\beta)}\circ y_{(\alpha,\beta)}\\
& =x_{(\alpha,\beta)}\circ (a^{2})_{(\alpha,\beta)}\circ y_{(\alpha,\beta)}\\
& =(xa^{2}y)_{(\alpha\wedge\alpha\wedge\alpha,\beta\vee\beta\vee\beta)}\\
&=a_{(\alpha,\beta)}.
\end{align*}
Hence $\underline{S}$ is intra-regular.

Conversely, let $\underline{S}$ is intra-regular and $a\in S.$ Then for any $\alpha,\beta\in [0,1],$ there exist elements $x_{(\gamma,\delta)},y_{(\eta,\theta)}\in\underline{S}$ such that%
\begin{align*}
a_{(\alpha,\beta)}& = x_{(\gamma,\delta)}\circ a_{(\alpha,\beta)}\circ a_{(\alpha,\beta)}\circ y_{(\eta,\theta)}=x_{(\gamma,\delta)}\circ (a^{2})_{(\alpha\wedge\alpha,\beta\vee\beta)}\circ y_{(\eta,\theta)}\\
& =x_{(\gamma,\delta)}\circ (a^{2})_{(\alpha,\beta)}\circ y_{(\eta,\theta)}=(xa^{2}y)_{(\gamma\wedge\alpha\wedge\eta,\delta\vee\beta\vee\theta)}.
\end{align*}
This implies that $a=xa^{2}y$ and $x,y\in S.$ Hence $S$ is intra-regular.
\end{proof}

\begin{theorem}
A semigroup $S$ is regular if and only if the semigroup $\underline{S}$ is regular.
\end{theorem}

\begin{proof}
Let $a_{(\alpha,\beta)}\in\underline{S}$ and $a\in S.$ Then there exists an element $x\in S$ such that $a=axa($since $S$ is regular$).$ So $x_{(\alpha,\beta)}\in\underline{S}.$ Then%
\begin{align*}
a_{(\alpha,\beta)}\circ x_{(\alpha,\beta)}\circ a_{(\alpha,\beta)}& =(axa)_{(\alpha\wedge\alpha\wedge\alpha,\beta\vee\beta\vee\beta)}=(axa)_{(\alpha,\beta)}=a_{(\alpha,\beta)}.
\end{align*}
Hence $\underline{S}$ is regular.

Conversely, let $\underline{S}$ is regular and $a\in S.$ Then for any $\alpha,\beta\in [0,1],$ there exists an element $x_{(\gamma,\delta)}\in\underline{S}$ such that%
\begin{align*}
a_{(\alpha,\beta)}& = a_{(\alpha,\beta)}\circ x_{(\gamma,\delta)}\circ a_{(\alpha,\beta)}=(axa)_{(\alpha\wedge\gamma\wedge\alpha,\beta\vee\delta\vee\beta)}=(axa)_{(\gamma\wedge\alpha,\beta\vee\delta)}.
\end{align*}
This implies that $a=axa$ and $x\in S.$ Hence $S$ is regular.
\end{proof}

\begin{lemma}
$\cite{K2}$ For a semigroup $S,$ the following conditions are equivalent:

$(i)$ $S$ is intra-regular$($regular$),$

$(ii)$ $L\cap R\subset LR($resp. $R\cap L=RL)$ holds for every left ideal $L$ and right ideal $R$ of $S.$
\end{lemma}

\begin{lemma}
$\cite{S}$ For a semigroup $S,$ the following conditions are equivalent:

$(i)$ $S$ is intra-regular$($regular$),$

$(ii)$ $A\cap B\subset A\circ B($resp. $B\cap A=B\circ A)$ holds for every intuitionistic fuzzy left ideal $A=(\mu_{A},\nu_{A})$ and intuitionistic fuzzy right ideal $B=(\mu_{B},\nu_{B})$ of $S.$
\end{lemma}

\begin{theorem}
For a semigroup $S,$ the following conditions are equivalent:

$(i)$ $S$ is intra-regular,

$(ii)$ $\underline{A}\cap\underline{B}\subset \underline{A}\circ\underline{B}$ for every intuitionistic fuzzy left ideal $A=(\mu_{A},\nu_{A})$ and intuitionistic fuzzy right ideal $B=(\mu_{B},\nu_{B})$ of $S.$
\end{theorem}

\begin{proof}
$(i)\Rightarrow (ii):$ Let $S$ be an intra-regular semigroup and $A=(\mu_{A},\nu_{A})$, $B=(\mu_{B},\nu_{B})$ are respectively intuitionistic fuzzy left and intuitionistic fuzzy right ideal of $S.$ By Theorem $3.10,$ $\underline{S}$ is an intra-regular semigroup. By Theorem $3.5,$ $\underline{A}$ and $\underline{B}$ are respectively left and right ideal of the semigroup $\underline{S}.$ Hence by Lemma $3.12,$ $\underline{A}\cap\underline{B}\subset \underline{A}\circ\underline{B}.$

$(ii)\Rightarrow (i):$ Let $A=(\mu_{A},\nu_{A})$ be an intuitionistic fuzzy left and $B=(\mu_{B},\nu_{B})$ be an intuitionistic fuzzy right ideal of $S$ such that $\underline{A}\cap\underline{B}\subset \underline{A}\circ\underline{B}.$ Let $x\in S.$ If $(\mu_{A}(x)=0$ or $\mu_{B}(x)=0)$ and $(\nu_{A}(x)=1$ or $\nu_{B}(x)=1),$ then $0=\mu_{A}(x)\wedge\mu_{B}(x)\leq(\mu_{A}\circ\mu_{B})(x),i.e.,(\mu_{A}\cap\mu_{B})(x)\leq(\mu_{A}\circ\mu_{B})(x)$ and $1=\nu_{A}(x)\vee\nu_{B}(x)\geq(\nu_{A}\circ\nu_{B})(x),i.e.,(\nu_{A}\cup\nu_{B})(x)\geq(\nu_{A}\circ\nu_{B})(x).$ Hence $A\circ B\supset A\cap B.$

Again if $(\mu_{A}(x)>0$ and $\nu_{A}(x)<1)$ and $(\mu_{B}(x)>0$ and $\nu_{B}(x)<1),$ then $x_{(\mu_{A}(x)\wedge\mu_{B}(x),\nu_{A}(x)\vee\nu_{B}(x))}\in\underline{A}$ and $x_{(\mu_{A}(x)\wedge\mu_{B}(x),\nu_{A}(x)\vee\nu_{B}(x))}\in\underline{B}.$

Hence $x_{(\mu_{A}(x)\wedge\mu_{B}(x),\nu_{A}(x)\vee\nu_{B}(x))}\in\underline{A}\cap \underline{B}\subset\underline{A}\circ \underline{B}\subseteq\underline{A\circ B}(cf.$ Proposition $3.1).$ Then $(\mu_{A}\circ\mu_{B})(x)\geq\mu_{A}(x)\wedge\mu_{B}(x)=(\mu_{A}\cap\mu_{B})(x)$ and $(\nu_{A}\circ\nu_{B})(x)\leq\nu_{A}(x)\vee\mu_{B}(x)=(\nu_{A}\cup\nu_{B})(x)$ which implies that $A\circ B\supseteq A\cap B.$ Hence by Lemma $3.13,$ $S$ is intra-regular.
\end{proof}

\begin{lemma}
$\cite{S}$ Let $A=(\mu_{A},\nu_{A})$ be an intuitionistic fuzzy right ideal and $B=(\mu_{B},\nu_{B})$ be an intuitionistic fuzzy left ideal of a semigroup $S.$ Then $A\circ B\subseteq A\cap B.$
\end{lemma}

\begin{theorem}
For a semigroup $S,$ the following conditions are equivalent:

$(i)$ $S$ is regular,

$(ii)$ $\underline{B}\cap\underline{A}= \underline{B}\circ\underline{A}$ for every intuitionistic fuzzy left ideal $A=(\mu_{A},\nu_{A})$ and intuitionistic fuzzy right ideal $B=(\mu_{B},\nu_{B})$ of $S.$
\end{theorem}

\begin{proof}
$(i)\Rightarrow (ii):$ Let $S$ be a regular semigroup and $A=(\mu_{A},\nu_{A})$, $B=(\mu_{B},\nu_{B})$ are respectively intuitionistic fuzzy left and intuitionistic fuzzy right ideal of $S.$ By Theorem $3.11,$ $\underline{S}$ is a regular semigroup. By Theorem $3.5,$ $\underline{A}$ and $\underline{B}$ are respectively left and right ideal of the semigroup $\underline{S}.$ Hence by Lemma $3.12,$ $\underline{A}\cap\underline{B}=\underline{A}\circ\underline{B}.$

$(ii)\Rightarrow (i):$ Form $(ii)\Rightarrow (i)$ we have $A\circ B\supseteq A\cap B.$ Again by Lemma $3.15,$ $A\circ B\subseteq A\cap B.$ Consequently, $A\circ B=A\cap B$ and hence by Lemma $3.13,$ $S$ is regular.
\end{proof}

\begin{theorem}
Let $A=(\mu_{A},\nu_{A})$ be a non-empty intuitionistic fuzzy subset of a semigroup $S.$ Then the following are equivalent:

$(i)$ $A=(\mu_{A},\nu_{A})$ is an intuitionistic fuzzy semiprime ideal of $S,$

$(ii)$ $\underline{A}$ is a semiprime ideal of $\underline{S}.$
\end{theorem}

\begin{proof}
$(i)\Rightarrow (ii):$ Let $A=(\mu_{A},\nu_{A})$ be an intuitionistic fuzzy semiprime ideal of $S.$ Then $\mu_{A}(x)\geq\mu_{A}(x^{2})$ and $\nu_{A}(x)\leq\nu_{A}(x^{2})\forall x\in S.$ Let $x_{(\alpha,\beta)}\circ x_{(\alpha,\beta)}\in\underline{A},i.e.,(x^{2})_{(\alpha,\beta)}\in\underline{A}.$ Then $\mu_{A}(x^{2})\geq\alpha$ and $\nu_{A}(x^{2})\leq\beta.$ Since $A=(\mu_{A},\nu_{A})$ is an intuitionistic fuzzy semiprime ideal of $S,$ so $\mu_{A}(x)\geq\mu_{A}(x^{2})\geq\alpha$ and $\nu_{A}(x)\leq\nu_{A}(x^{2})\leq\beta,$ which implies that $x_{(\alpha,\beta)}\in\underline{A}.$ Hence $\underline{A}$ is a semiprime ideal of $\underline{S}.$

$(ii)\Rightarrow (i):$ Let $\underline{A}$ be a semiprime ideal of $\underline{S}.$ Let $\mu_{A}(x^{2})=\alpha$ and $\nu_{A}(x^{2})=\beta.$ Then $(x^{2})_{(\alpha,\beta)}\in\underline{A},i.e.,x_{(\alpha,\beta)}\circ x_{(\alpha,\beta)}\in\underline{A}$ implies that $x_{(\alpha,\beta)}\in\underline{A}($since $\underline{A}$ is a semiprime ideal of $\underline{S}).$ Then $\mu_{A}(x)\geq\alpha$ and $\nu_{A}(x)\leq\beta$ which implies that $\mu_{A}(x)\geq\mu_{A}(x^{2})$ and $\nu_{A}(x)\leq\nu_{A}(x^{2}).$ Hence $A=(\mu_{A},\nu_{A})$ is an intuitionistic fuzzy semiprime ideal of $S.$
\end{proof}

\begin{lemma}
$\cite{S}$ For a semigroup $S$ following conditions are equivalent:

$(i)$ $S$ is an intra-regular semigroup,

$(ii)$ every intuitionistic fuzzy ideal $A=(\mu_{A},\nu_{A})$ of $S$ is an intuitionistic fuzzy semiprime ideal of $S.$
\end{lemma}

\begin{theorem}
For any non-empty intuitionistic fuzzy subset $A=(\mu_{A},\nu_{A})$ of an intra-regular semigroup $S$ the following conditions are equivalent:

$(i)$ $A=(\mu_{A},\nu_{A})$ is an intuitionistic fuzzy ideal of $S,$

$(ii)$ $\underline{A}$ is a semiprime ideal of $\underline{S}.$
\end{theorem}

\begin{proof}
$(i)\Rightarrow (ii):$ Let $A=(\mu_{A},\nu_{A})$ be an intuitionistic fuzzy ideal of an intra-regular semigroup $S.$ Then by Lemma $3.18,$ $A=(\mu_{A},\nu_{A})$ is an intuitionistic fuzzy semiprime ideal of $S$ and hence by Theorem $3.17,$ $\underline{A}$ is a semiprime ideal of $\underline{S}.$

$(ii)\Rightarrow (i):$ Let $\underline{A}$ is a semiprime ideal of $\underline{S}.$ Then by Theorem $3.17,$ $A=(\mu_{A},\nu_{A})$ is an intuitionistic fuzzy semiprime ideal of $S$ and hence $A=(\mu_{A},\nu_{A})$ is an intuitionistic fuzzy ideal of $S.$
\end{proof}

\begin{theorem}
Let $A=(\mu_{A},\nu_{A})$ be a non-empty intuitionistic fuzzy subset of a semigroup $S.$ Then the following are equivalent:

$(i)$ $A=(\mu_{A},\nu_{A})$ is an intuitionistic fuzzy prime ideal of $S,$

$(ii)$ $\underline{A}$ is a prime ideal of $\underline{S}.$
\end{theorem}

\begin{proof}
$(i)\Rightarrow (ii):$ Let $A=(\mu_{A},\nu_{A})$ be an intuitionistic fuzzy prime ideal of $S.$ Then $\mu_{A}(xy)=\max\{\mu_{A}(x),\mu_{A}(y)\}$ and $\nu_{A}(xy)=\min\{\nu_{A}(x),\nu_{A}(y)\}\forall x,y\in S.$ Let $x_{(\alpha,\beta)}\circ y_{(\alpha,\beta)}\in\underline{A},i.e.,(xy)_{(\alpha\wedge\alpha,\beta\vee\beta)}=(xy)_{(\alpha,\beta)}\in\underline{A}.$ Then $\mu_{A}(xy)\geq\alpha$ and $\nu_{A}(xy)\leq\beta$ which implies that $\max\{\mu_{A}(x),\mu_{A}(y)\}\geq\alpha$ and $\min\{\nu_{A}(x),\nu_{A}(y)\}\leq\beta.$ Then $(\mu_{A}(x)\geq\alpha$ or $\mu_{A}(y)\geq\alpha)$ and $(\nu_{A}(x)\leq\beta$ or $\nu_{A}(y)\leq\beta),i.e., (\mu_{A}(x)\geq\alpha$ and $\nu_{A}(x)\leq\beta)$ or $(\mu_{A}(y)\geq\alpha$ and $\nu_{A}(y)\leq\beta),i.e., x_{(\alpha,\beta)}\in\underline{A}$ or $y_{(\alpha,\beta)}\in\underline{A}.$ Hence $\underline{A}$ is a prime ideal of $\underline{S}.$

$(ii)\Rightarrow (i):$ Let $\underline{A}$ be a prime ideal of $\underline{S}.$ Let $\mu_{A}(xy)=\alpha$ and $\nu_{A}(xy)=\beta.$ Then $x_{(\alpha,\beta)}\circ y_{(\alpha,\beta)}=(xy)_{(\alpha,\beta)}\in\underline{A}$ implies that $x_{(\alpha,\beta)}\in\underline{A}$ or $y_{(\alpha,\beta)}\in\underline{A}($since $\underline{A}$ is a prime ideal of $\underline{S}).$ Then $(\mu_{A}(x)\geq\alpha$ and $\nu_{A}(x)\leq\beta)$ or $(\mu_{A}(y)\geq\alpha$ and $\nu_{A}(y)\leq\beta),i.e., (\mu_{A}(x)\geq\alpha$ or $\mu_{A}(y)\geq\alpha)$ and $(\nu_{A}(y)\leq\beta$ or $\nu_{A}(y)\leq\beta),i.e.,\max\{\mu_{A}(x),\mu_{A}(y)\}\\\geq\alpha$ and $\min\{\nu_{A}(x),\nu_{A}(y)\}\leq\beta,i.e.,\max\{\mu_{A}(x),\mu_{A}(y)\}\geq\mu_{A}(xy)$ and $\min\{\nu_{A}(x),\nu_{A}\\(y)\}\leq\nu_{A}(xy).........(A).$ Since $\underline{A}$ is a prime ideal of $\underline{S},$ so $\underline{A}$ is an ideal of $\underline{S}.$ So, by Theorem $3.5,$ $A=(\mu_{A},\nu_{A})$ is an intuitionistic fuzzy ideal of $S.$ Then $\mu_{A}(xy)\geq\max\{\mu_{A}(x),\mu_{A}(y)\}$ and $\nu_{A}(xy)\leq\min\{\nu_{A}(x),\nu_{A}(y)\}\forall x,y\in S........(B).$ Combining $(A)$ and $(B)$ we have, $\mu_{A}(xy)=\max\{\mu_{A}(x),\mu_{A}(y)\}$ and $\nu_{A}(xy)=\min\{\nu_{A}(x),\nu_{A}(y)\}\forall x,y\in S.$ Hence $A=(\mu_{A},\nu_{A})$ is an intuitionistic fuzzy prime ideal of $S.$
\end{proof}


\end{document}